\documentclass[12pt,letterpaper]{article}
\usepackage{amsthm,amssymb}
\usepackage{amsfonts}
\makeatletter
\newtheorem{thm}{\textbf Theorem}[section]
\newtheorem{lem}[thm]{\textbf Lemma}
\newtheorem{rem}[thm]{\textbf Remark}

\newcommand{\be}{\begin{eqnarray*}}
\newcommand{\ee}{\end{eqnarray*}}

\begin{document}

\title{\bf  On the generalized Feynman-Kac transformation for nearly symmetric Markov processes}
  \author{\large  Li Ma  \qquad Wei Sun
            \\ Department of Mathematics and Statistics\\
             Concordia University, Canada}
   \date{}
\maketitle

\begin{abstract}
\noindent Suppose $X$ is a right process which is associated with
a non-symmetric Dirichlet form $(\mathcal{E},D(\mathcal{E}))$ on
$L^{2}(E;m)$. For $u\in D(\mathcal{E})$, we have Fukushima's
decomposition:
$\tilde{u}(X_{t})-\tilde{u}(X_{0})=M^{u}_{t}+N^{u}_{t}$. In this
paper, we investigate the strong continuity of the generalized
Feynman-Kac semigroup defined by
$P^{u}_{t}f(x)=E_{x}[e^{N^{u}_{t}}f(X_{t})]$. Let
$Q^{u}(f,g)=\mathcal{E}(f,g)+\mathcal{E}(u,fg)$ for $f,g\in
D(\mathcal{E})_{b}$. Denote by $J_1$ the dissymmetric part of the
jumping measure $J$ of $(\mathcal{E},D(\mathcal{E}))$. Under the
assumption that $J_1$ is finite, we show that
$(Q^{u},D(\mathcal{E})_{b})$ is lower semi-bounded if and only if
there exists a constant $\alpha_0\ge 0$ such that
$\|P^{u}_{t}\|_2\leq e^{\alpha_0 t}$ for every $t>0$. If one of
these conditions holds, then $(P^{u}_{t})_{t\geq0}$ is strongly
continuous on $L^{2}(E;m)$. If $X$ is equipped with a differential
structure, then this result also holds without assuming that $J_1$
is finite.

\end{abstract}
\section[short title]{Introduction }
Let $E$ be a metrizable Lusin space and
$X=((X_{t})_{t\geq0},(P_{x})_{x\in E})$ be a right (continuous
strong Markov) process on $E$ (cf. \cite[IV, Definition
1.8]{MR92}). Suppose that $X$ is associated with a (non-symmetric)
Dirichlet form $(\mathcal{E},D(\mathcal{E}))$ on $L^2(E;m)$, where
$m$ is a $\sigma$-finite measure on the Borel $\sigma$-algebra
${\mathcal{B}}(E)$ of $E$. Then, by \cite[IV, Theorem 6.7]{MR92}
(cf. also \cite[Theorem 3.22]{F01}),
$(\mathcal{E},D(\mathcal{E}))$ is quasi-regular. Moreover,
$(\mathcal{E},D(\mathcal{E}))$ is quasi-homeomorphic to a regular
Dirichlet form (see \cite{C04}). We refer the reader to
\cite{MR92} and \cite{Fu94} for the theory of Dirichlet forms. The
notations and terminologies of this paper follow \cite{MR92} and
\cite{Fu94}.

Let $u\in D(\mathcal{E})$. Then, we have Fukushima's decomposition
(cf. \cite[VI, Theorem 2.5]{MR92})
$$\tilde{u}(X_{t})-\tilde{u}(X_{0})=M^{u}_{t}+N^{u}_{t},$$
where $\tilde{u}$ is a quasi-continuous $m$-version of $u$,
$M_{t}^u$ is a square integrable martingale additive functional
(MAF) and $N_{t}^u$ is a continuous additive functional (CAF) of
zero energy. For $x\in E$, denote by $E_x$ the expectation with
respect to (w.r.t.) $P_x$. Define the generalized Feynman-Kac
transformation
$$
P^{u}_{t}f(x)=E_{x}[e^{N^{u}_{t}}f(X_{t})],\ \ f\ge 0\ {\rm and}\
t\ge 0.
$$
In this paper, we will investigate the strong continuity of the
semigroup $(P^{u}_{t})_{t\geq0}$ on $L^2(E;m)$.

The strong continuity of generalized Feynman-Kac semigroups for
symmetric Markov processes has been studied extensively by many
people. Note that in general $(N_{t}^u)_{t\ge 0}$ is not of finite
variation (cf. \cite[Example 5.5.2]{Fu94}). Hence the classical
results of Albeverio and Ma given in \cite{AM1} do not apply
directly. Under the assumption that $X$ is the standard
$d$-dimensional Brownian motion, $u$ is a bounded continuous
function on $\mathbf{R}^d$ and $|\nabla u|^{2}$ belongs to the
Kato class, Glover et al. proved in \cite{G94} that
$(P_{t}^{u})_{t\ge 0}$ is a strongly continuous semigroup on
$L^{2}({\mathbf{R}^d};dx)$. Moreover, they gave an explicit
representation for the closed quadratic form corresponding to
$(P_{t}^{u})_{t\ge 0}$. In \cite{TSZ01}, Zhang generalized the
results of \cite{G94} to symmetric L\'evy processes on
$\mathbf{R}^d$ and removed the assumption that $u$ is bounded
continuous. Furthermore, Z.Q. Chen and Zhang established in
\cite{C02} the corresponding results for general symmetric Markov
processes via Girsanov transformation. They proved that if
$\mu_{\langle u\rangle}$, the energy measure of $u$, is a measure
of the Kato class, then $(P_{t}^{u})_{t\ge 0}$ is a strongly
continuous semigroup on $L^{2}(E;m)$. Also, they characterized the
closed quadratic form corresponding to $(P_{t}^{u})_{t\ge 0}$. In
\cite{FK04}, Fitzsimmons and Kuwae established the strong
continuity of $(P_{t}^{u})_{t\ge 0}$ under the assumption that $X$
is a symmetric diffusion process and $\mu_{\langle u\rangle}$ is a
measure of the Hardy class. Furthermore, Z.Q. Chen et al.
established in \cite{C08} the strong continuity of
$(P_{t}^{u})_{t\ge 0}$ for general symmetric Markov processes
under the assumption that $\mu_{\langle u\rangle}$ is a measure of
the Hardy class.

All the results mentioned above give sufficient conditions for
$(P_{t}^{u})_{t\ge 0}$ to be strongly continuous, where
$\mu_{\langle u\rangle}$ is assumed to be of either the Kato class
or the Hardy class. In \cite{CS06}, under the assumption that $X$
is a symmetric diffusion process, C.Z. Chen and Sun showed that
the semigroup $(P_{t}^{u})_{t\ge 0}$ is strongly continuous on
$L^2(E;m)$ if and only if the bilinear form
$({{Q}}^{u},D(\mathcal{E})_{b})$ is lower semi-bounded. Here and
henceforth
\begin{equation}\label{ap101}Q^{u}(f,g):=\mathcal{E}(f,g)+\mathcal{E}(u,fg),\ \ f,g\in
D(\mathcal{E})_{b}:=D(\mathcal{E})\cap L^{\infty}(E;m).
\end{equation}
Furthermore, C.Z. Chen et al. generalized this result to general
symmetric Markov processes in \cite{CSM07}. In \cite{C09}, Z.Q.
Chen et al. studied general perturbations of symmetric Markov
processes and gave another proof for the equivalence of the strong
continuity of $(P_{t}^{u})_{t\ge 0}$ and the lower
semi-boundedness of $({{Q}}^{u},D(\mathcal{E})_{b})$.

The aim of this paper is to study the strong continuity problem of
generalized Feynman-Kac semigroups for nearly symmetric Markov
processes. Note that many useful tools of symmetric Dirichlet
forms, e.g. time reversal and Lyons-Zheng decomposition, do not
apply well to the non-symmetric setting.  That makes the problem
more difficult. Also, we would like to point out that the Girsanov
transformed process of $X$ induced by $M^u_t$ and the Girsanov
transformed process of $\hat{X}$ induced by $\hat{M}^u_t$ are not
in duality in general (cf. \cite{CS09}), where $\hat{X}$ is the
dual process of $X$ and $\hat{M}^u_t$ is the martingale part of
$\tilde{u}(\hat{X}_t)-\tilde{u}(\hat{X}_0)$. The method of this
paper is inspired by \cite{CSM07} and \cite{C09}. We will combine
the $h$-transform method of \cite{CSM07} and the localization
method used in \cite{C09}. It is worth to point out that the
Beurling-Deny formula given in \cite{HC06} and LeJan's
transformation rule developed in \cite{HZC08} play a crucial role
in this paper.

Denote by $J$ and $K$ the jumping and killing measures of
$(\mathcal{E},D(\mathcal{E}))$, respectively. Write
$\hat{J}(dx,dy)=J(dy,dx)$. Denote by $J_1:=(J-\hat{J})^+$ the
positive part of the Jordan decomposition of $J-\hat{J}$. $J_1$ is
called the dissymmetric part of $J$. Note that $J_0:=J-J_1$ is the
largest symmetric $\sigma$-finite positive measure dominated by
$J$. Denote by $d$ the diagonal of the product space $E\times E$;
and denote by $\|\cdot\|_2$ and $(\cdot,\cdot)_m$ the norm and
inner product of $L^2(E;m)$, respectively.

Now we can state the main results of the paper.
\begin{thm}\label{th1.1}
Suppose that $X$ is a right process which is associated with a
(non-symmetric) Dirichlet form $(\mathcal{E},D(\mathcal{E}))$ on
$L^{2}(E;m)$. Let $u\in D({\cal E})$. Assume that $J_{1}(E\times
E\backslash d)<\infty$.
Then the following two conditions are equivalent:\\
(i) There exists a constant $\alpha_{0}\geq0$ such that
$$
  Q^{u}(f,f)\geq -\alpha_0(f,f)_m,\ \ \forall f\in
D(\mathcal{E})_{b}.
$$
 (ii) There exists a constant $\alpha_{0}\geq0$ such that
$$
\|P^{u}_{t}\|_2\leq e^{\alpha_{0}t},\ \ \forall t>0.$$
Furthermore, if one of these conditions holds, then the semigroup
$(P^{u}_{t})_{t\geq0}$ is strongly continuous on $L^{2}(E;m)$.
\end{thm}

\begin{thm}\label{th1.2}
Let $U$ be an open set of $\mathbf{R^d}$ and $m$ be a positive
Radon measure on $U$ with ${\rm supp}[m]=U$. Suppose that $X$ is a
right process which is associated with a (non-symmetric) Dirichlet
form $(\mathcal{E},D(\mathcal{E}))$ on $L^{2}(U;m)$ such that
$C^{\infty}_0(U)$ is dense in $D(\mathcal{E})$. Then the
conclusions of Theorem \ref{th1.1} remain valid without assuming
that $J_{1}(E\times E\backslash d)<\infty$.
\end{thm}
The rest of this paper is organized as follows. In Section 2, we
give the proof of Theorems \ref{th1.1} and \ref{th1.2}. In Section
3, we give two examples, one is finite-dimensional and the other
one is infinite-dimensional.
\section[short title]{Proof of the Main Results}\setcounter{equation}{0}

In this section, we will prove Theorems \ref{th1.1} and
\ref{th1.2}. By quasi-homeomorphism, we assume without loss of
generality that $X$ is a Hunt process and
$(\mathcal{E},D(\mathcal{E}))$ is a regular (non-symmetric)
Dirichlet form
 on $L^{2}(E;m)$, where $E$ is a locally compact
separable metric space and $m$ is a positive Radon measure on $E$
with ${\rm supp}[m]=E$. We denote by $\Delta$ and $\zeta$ the
cemetery and lifetime of $X$, respectively. It is known that every
$f\in D(\mathcal{E})$ has a quasi-continuous $m$-version. To
simplify notation, we still denote this version by $f$.

Let $u\in D(\mathcal{E})$. By \cite[III, Proposition 1.5]{MR92},
there exists $|u|_E\in D({\cal E})$ such that $|u|_E\ge |u|$
$m-a.e.$ on $E$ and ${\cal E}_1(|u|_E,w)\ge 0$ for all $w\in
D({\cal E})$ with $w\ge 0$ $m-a.e.$ on $E$. Similar to
\cite[Theorems 2.2.1 and 2.2.2]{Fu94}, we can show that there
exists a positive Radon measure $\eta_u$ on $E$ such that
\begin{equation}\label{dd}
{\cal E}_1(|u|_E,w)=\int_Ewd\eta_u,\ \ w\in D({\cal E}).
\end{equation}
Define \begin{equation}\label{u}u^{*}:=u+|u|_E.\end{equation}
Then, $u^{*}$ has a quasi-continuous $m$-version which is
nonnegative q.e. on $E$. Moreover, there exists an ${\cal E}$-nest
$\{F_n\}_{n\in\mathbf{N}}$ consisting of compact sets of $E$ such
that $u^{*}$ is continuous and hence bounded on $F_n$ for each
$n\in\mathbf{N}$. Define $\tau_{F_n}=\inf\{t>0| X_{t}\in
F^c_{n}\}$. By \cite[IV, Proposition 5.30]{MR92},
$\lim_{n\rightarrow\infty}\tau_{F_n}=\zeta$ $P_x$-a.s. for q.e.
$x\in E$.

Let $(N,H)$ be a L\'evy system of $X$ and $\nu$ be the Revuz
measure of $H$.  Define \begin{equation}\label{ap11}
B_{t}=\sum_{s\leq
t}\left[e^{(u^{*}(X_{s-})-u^{*}(X_{s}))}-1-(u^{*}(X_{s-})-u^{*}(X_{s}))\right].\end{equation}
Note that for any $M>0$ there exists $C_M>0$ such that
$(e^x-1-x)\le C_Mx^2$ for all $x$ satisfying $x\le M$. Since
$(u^{*}(X_{t-}))_{t\ge 0}$ is locally bounded, $(u^{*}(X_t))_{t\ge
0}$ is nonnegative and $M^{-u^{*}}$ is a square integrable
martingale for q.e. $x\in E$, hence $(B_t)_{t\ge 0}$ is locally
integrable on $[0, \zeta)$ for q.e. $x\in E$. Here and henceforth
the phrase ``on $[0, \zeta)$" is understood as ``on the optional
set $[[0, \zeta))$ of interval type" in the sense of \cite[Chap.
VIII, §3]{He}. By \cite[(A.3.23)]{Fu94}, one finds that the dual
predictable projection of $(B_t)_{t\ge 0}$ is given by
$$
B^{p}_{t}=\int_{0}^{t}\int_{
E_{\triangle}}[e^{(u^{*}(X_{s})-u^{*}(y))}-1-(u^{*}(X_{s})-u^{*}(y))]N(X_{s},dy)dH_{s}.
$$

We set
\begin{equation}\label{add31}
   M_{t}^{d}=B_{t}-B_{t}^{p}
\end{equation}
and denote
\begin{equation}\label{add32}
M_{t}=M_{t}^{-u^{*}}+M_{t}^{d}. \end{equation} Note that for any
$M>0$ there exists $D_M>0$ such that $(e^x-1-x)^2\le D_Mx^2$ for
all $x$ satisfying $x\le M$. Since $(u^{*}(X_{t-}))_{t\ge 0}$ is
locally bounded, $(u^{*}(X_t))_{t\ge 0}$ is nonnegative and
$M^{-u^{*}}$ is a square integrable martingale for q.e. $x\in E$,
hence $(M^d_{t})_{t\ge 0}$ is a locally square integrable MAF on
$[0, \zeta)$ for q.e. $x\in E$ by \cite[Theorem 7.40]{He}.
Therefore $(M_{t})_{t\ge 0}$ is a locally square integrable MAF on
$[0, \zeta)$ for q.e. $x\in E$. We denote the Revuz measure of
$(<M>_t)_{t\ge 0}$ by $\mu_{<M>}$ (cf. \cite[Remark 2.2]{C08a}).

Let $M^{-u^{*},c}_{t}$ be the continuous part of $M^{-u^{*}}_{t}$.
Define
\begin{equation}\label{add33}
A^{-u^{*}}_{t}=B^{p}_{t}+\frac{1}{2}<M^{-u^{*},c}>_{t}.
\end{equation}
Then $(A^{-u^{*}}_{t})_{t\ge 0}$ is a positive CAF (PCAF). Denote
by $\mu_{-u^{*}}$ the Revuz measure of $(A^{-u^{*}}_{t})_{t\ge
0}$. Then
\begin{eqnarray}\label{add34}
\mu_{-u^{*}}(dx)&=&\int_{E_{\triangle}}
        [e^{(u^{*}(x)
        -u^{*}(y))}-1-(u^{*}(x)-u^{*}(y))]N(x,dy)\nu(dx)\nonumber\\
        & &\ \ \ \ +\frac{1}{2}\mu_{<M^{-u^{*},c}>}(dx).
\end{eqnarray}
Define
\begin{equation}\label{ap1}
\mu_{-u}:=\mu_{-u^{*}}+\eta_u-|u|_Em
\end{equation}
and
$$
\mu'_{-u}:=\mu_{-u^{*}}+\eta_u+|u|_Em.
$$

Recall that a smooth measure $\mu$ is said to be of the Kato class
if $$\lim_{t\rightarrow 0}\inf_{{\rm Cap}(N)=0}\sup_{x\in
E-N}E_x[A^{\mu}_t]=0,
$$
where $(A^{\mu}_t)_{t\ge 0}$ is the PCAF associated with $\mu$.
Denote by $S_{K}$ the Kato class of smooth measures. Similar to
\cite[Theorem 2.4]{AM2}, we can show that there exists an ${\cal
E}$-nest $\{F'_n\}_{n\in\mathbf{N}}$ consisting of compact sets of
$E$ such that $I_{F'_{n}}(\mu_{<M>}+\mu'_{-u})\in S_{K}$. To
simplify notation, we still use $F_n$ to denote $F_n\cap F'_n$ for
$n\in\mathbf{N}$. Let $E_{n}$ be the fine interior of $F_{n}$
w.r.t. $X$. Define $D(\mathcal{E})_{n}:=\{f\in D(\mathcal{E})|f=0
\ {\rm q.e.\ on}\ E_{n}^{c}\}$, $\tau_{E_n}=\inf\{t>0| X_{t}\in
E^c_{n}\}$ and
$$\bar{P}_{t}^{u,n}f(x):=E_{x}[e^{M^{-u^{*}}_{t}-N_t^{|u|_E}}f(X_t);t<\tau_{E_n}].$$

\subsection[short title]{The bilinear form associated with
${(\bar{P}^{u,n}_{t})_{t\geq0}}$ on ${L^2(E_n;m)}$}

\noindent For $n\in\mathbf{N}$, we define the bilinear form
$(\bar{Q}^{u,n},D(\mathcal{E})_n)$ by
\begin{equation}\label{e2.8}
\bar{Q}^{u,n}(f,g)=\mathcal{E}(f,g)-\int_{E}gd\mu_{<M^{f},M>}-\int_{E}fgd\mu_{-u},\
\ f,g\in D(\mathcal{E})_n.
\end{equation}
By \cite[Lemma 4.3]{C03},  for every $\varepsilon>0$, there exists
a constant $A^n_{\varepsilon}>0$ such that
$$
\int_{E}w^{2}d(\mu_{<M>}+\mu'_{-u})\leq \varepsilon
\mathcal{E}(w,w)+A^n_{\varepsilon}\|w\|^{2}_{2},\ \ w\in
D(\mathcal{E})_n.
$$
Suppose that $|{\cal E}(f,g)|\le k_1{\cal
E}_1(f,f)^{\frac{1}{2}}{\cal E}_1(g,g)^{\frac{1}{2}}$ for all
$f,g\in D(\mathcal{E})$ and some constant $k_1>0$. Then
\begin{eqnarray}\label{e2.9} |\bar{Q}^{u,n}(f,g)|&\leq&k_{1}\mathcal{E}_{1}(f,f)^{\frac{1}{2}}\mathcal{E}_{1}(g,g)^{\frac{1}{2}}
       +\left(\int_{E}d\mu_{<M^{f}>}\right)^{\frac{1}{2}}\left(\int_{E}g^{2}d\mu_{<M>}\right)^{\frac{1}{2}}\nonumber\\
\nonumber &&+
       \left(\int_{E}f^{2}d\mu'_{-u}\right)^{\frac{1}{2}}\left(\int_{E}g^{2}d\mu'_{-u}\right)^{\frac{1}{2}}\\
\nonumber
&\leq&k_{1}\mathcal{E}_{1}(f,f)^{\frac{1}{2}}\mathcal{E}_{1}(g,g)^{\frac{1}{2}}
       +(\max(\varepsilon, A^n_{\varepsilon}))^{\frac{1}{2}}[2\mathcal{E}(f,f)]^{\frac{1}{2}}\mathcal{E}_{1}(g,g)^{\frac{1}{2}}\\
\nonumber&&+\max(\varepsilon,A^n_{\varepsilon})\cdot\mathcal{E}_{1}(f,f)^{\frac{1}{2}}\mathcal{E}_{1}(g,g)^{\frac{1}{2}}\\
&\leq&\theta_n\mathcal{E}_{1}(f,f)^{\frac{1}{2}}\mathcal{E}_{1}(g,g)^{\frac{1}{2}},
\end{eqnarray}
where
$\theta_n:=(k_{1}+\sqrt{2\max(\varepsilon,A^n_{\varepsilon})}+\max(\varepsilon,A^n_{\varepsilon}))$.

Fix an $\varepsilon<(\sqrt{2}-1)/(\sqrt{2}+1)$ and set
$\alpha_n:=2A^n_{\varepsilon}$. Then
\begin{eqnarray}\label{e2.10}
 \bar{Q}^{u,n}_{\alpha_n}(f,f)&:=&\bar{Q}^{u,n}(f,f)+\alpha_n(f,f)\nonumber\\
\nonumber&\geq&\mathcal{E}(f,f)-\left(\int_{E}d\mu_{<M^{f}>}\right)^{\frac{1}{2}}\left(\int_{E}f^{2}d\mu_{<M>}\right)^{\frac{1}{2}}\\
& &-\int_{E}f^{2}d\mu'_{-u}+\alpha_n(f,f)\nonumber\\
\nonumber &\geq&\mathcal{E}(f,f)-(\varepsilon\
        \mathcal{E}(f,f)+A^n_{\varepsilon}\|f\|^{2}_{2})^{\frac{1}{2}}[2\mathcal{E}(f,f)]^{\frac{1}{2}}\nonumber\\
        & &
        \ \ \ \ \ \ -(\varepsilon\
        \mathcal{E}(f,f)+A^n_{\varepsilon}\|f\|^{2}_{2})+\alpha_n(f,f)\nonumber\\
\nonumber&\geq&\mathcal{E}(f,f)-\frac{1}{\sqrt{2}}((1+\varepsilon)\mathcal{E}(f,f)+A^n_{\varepsilon}\|f\|^{2}_{2})\nonumber\\
& &\ \ \ \ \ \ -(\varepsilon\
        \mathcal{E}(f,f)+A^n_{\varepsilon}\|f\|^{2}_{2})+\alpha_n(f,f)\nonumber\\
     &\geq&\frac{\sqrt{2}-1-(\sqrt{2}+1)\varepsilon}{\sqrt{2}}\mathcal{E}(f,f)
        +\frac{(\sqrt{2}-1)A^n_{\varepsilon}}{\sqrt{2}}\|
        f\|^{2}_{2}.
\end{eqnarray}

By (\ref{e2.9}), (\ref{e2.10}) and \cite[I, Proposition
3.5]{MR92}, we know that
$(\bar{Q}^{u,n}_{\alpha_n},D(\mathcal{E}))$ is a coercive closed
form on $L^{2}(E_n;m)$.

\begin{thm}\label{th2.1} For each $n\in\mathbf{N}$,
$(\bar{P}^{u,n}_{t})_{t\geq0}$ is a strongly continuous semigroup
of bounded operators on $L^{2}(E_n;m)$ with
$\|\bar{P}^{u,n}_{t}\|_{2}\leq e^{\beta_{n}t}$ for every $t>0$ and
some constant $\beta_n>0$. Moreover, the coercive closed form
associated with $(e^{-\beta_nt}\bar{P}^{u,n}_{t})_{t\geq0}$ is
given by $(\bar{Q}^{u,n}_{\beta_n},D(\mathcal{E})_n)$.
\end{thm}
\begin{proof}
The proof is much similar to that of \cite[Theorem 1.1]{FK04},
which is based on a key lemma (see \cite[Lemma 3.2]{FK04}) and a
remarkable localization method. In fact, the proof of our Theorem
\ref{th2.1} is simpler since $I_{F_n}(\mu_{<M>}+\mu'_{-u})$ is of
the Kato class instead of the Hardy class and there is no time
reversal part in the semigroup $(\bar{P}^{u,n}_{t})_{t\geq0}$. We
omit the details of the proof here and only give the following key
lemma, which is the counterpart of \cite[Lemma 3.2]{FK04}.
\end{proof}

\begin{lem}\label{lem3.1}
Let $(L^{\bar{Q}^{u,n}}, D(L^{\bar{Q}^{u.n}}))$ be the generator
of $(\bar{Q}^{u,n},D(\mathcal{E})_n)$. Then, for any $f\in
D(L^{\bar{Q}^{u,n}})$, we have
\begin{eqnarray}\label{add2}
f(X_{t})e^{{M^{-u^{*}}_{t}-N_t^{|u|_E}}}
&=&f(X_{0})+\int^{t}_{0}e^{{M^{-u^{*}}_{s-}-N_{s-}^{|u|_E}}}dM^{f}_{s}\nonumber\\
&
&+\int^{t}_{0}e^{{M^{-u^{*}}_{s-}-N_{s-}^{|u|_E}}}f(X_{s-})dM_{s}\nonumber\\
&
&+\int^{t}_{0}e^{{M^{-u^{*}}_{s-}-N_{s-}^{|u|_E}}}L^{\bar{Q}^{u,n}}f(X_{s})ds
\end{eqnarray}
$P_m$-a.s. on $\{t<\tau_{E_n}\}$.
\end{lem}
\begin{proof} Let $f\in
D(L^{\bar{Q}^{u,n}})$ and $g\in D(\mathcal{E})_n$. Then, by
(\ref{e2.8}), we get
\begin{eqnarray}\label{dd2}
\mathcal{E}(f,g)
&=&\bar{Q}^{u,n}(f,g)+\int_{E}gd\mu_{<M^{f},M>}+\int_{E}fgd\mu_{-u}\nonumber\\
&=&-(L^{\bar{Q}^{u,n}}f,g)+\int_{E}gd\mu_{<M^{f},M>}+\int_{E}fgd\mu_{-u}.
\end{eqnarray}
By (\ref{dd}), (\ref{dd2}) and \cite[Theorem 5.2.7]{oshima}, we
find that $(N_t^{|u|_E})_{t\ge 0}$ is a CAF of bounded variation
and
\begin{eqnarray*}
N^{f}_{t} &=&\int_{0}^{t}L^{\bar{Q}^{u,n}}f(X_{s})ds
-<M^{f},M>_{t}-\int_{0}^{t}f(X_{s})d(A^{-u^{*}}_{s}-N_s^{|u|_E})
\end{eqnarray*}
for $t<\tau_{E_n}$. Therefore, for $t<\tau_{E_n}$, we have
\begin{eqnarray}\label{add1}
f(X_{t})-f(X_{0})
&=&M^{f}_{t}+N^{f}_{t}\nonumber\\
&=&M^{f}_{t}+\int_{0}^{t}L^{\bar{Q}^{u,n}}f(X_{s})ds
-<M^{f},M>_{t}\nonumber\\
& &-\int_{0}^{t}f(X_{s})d(A^{-u^{*}}_{s}-N_s^{|u|_E}).
\end{eqnarray}

By It\^{o}'s formula (cf. \cite[II, Theorem 33]{P}), (\ref{add1})
and (\ref{add31})-(\ref{add33}), we obtain that for $t<\tau_{E_n}$
\begin{eqnarray}\label{add3}
f(X_{t})\hskip-0.5cm& &\hskip-0.5cm
e^{{M^{-u^{*}}_{t}-N_t^{|u|_E}}}\nonumber\\
&=&f(X_{0})+\int^{t}_{0}e^{{M^{-u^{*}}_{s-}-N_{s-}^{|u|_E}}}df(X_{s})+\int^{t}_{0}e^{{M^{-u^{*}}_{s-}-N_{s-}^{|u|_E}}}f(X_{s-})d({M^{-u^{*}}_{s}-N_s^{|u|_E}})\nonumber\\
&&+\frac{1}{2}\int^{t}_{0}e^{{M^{-u^{*}}_{s-}-N_{s-}^{|u|_E}}}f(X_{s-})d<M^{-u^{*},c}>_{s}+\int^{t}_{0}e^{{M^{-u^{*}}_{s-}-N_{s-}^{|u|_E}}}d<M^{f,c},M^{-u^{*},c}>_{s}\nonumber\\
&&+\sum_{s\leq
t}[f(X_{s})e^{{M^{-u^{*}}_{s}-N_s^{|u|_E}}}-f(X_{s-})e^{{M^{-u^{*}}_{s-}-N_{s-}^{|u|_E}}}\nonumber\\
& &\ \ \ \ \ \ -e^{{M^{-u^{*}}_{s-}-N_{s-}^{|u|_E}}}\triangle
f(X_{s})-f(X_{s-})e^{{M^{-u^{*}}_{s-}-N_{s-}^{|u|_E}}}\triangle
M^{-u^{*}}_{s}]\nonumber\\
&=&f(X_{0})+\int^{t}_{0}e^{{M^{-u^{*}}_{s-}-N_{s-}^{|u|_E}}}dM^{f}_{s}+\int^{t}_{0}e^{{M^{-u^{*}}_{s-}-N_{s-}^{|u|_E}}}L^{\bar{Q}^{u,n}}f(X_{s})ds
   \nonumber\\
&&-\int^{t}_{0}e^{{M^{-u^{*}}_{s-}-N_{s-}^{|u|_E}}}d<M^{f},M>_{s}
-\int_{0}^{t}e^{{M^{-u^{*}}_{s-}-N_{s-}^{|u|_E}}}f(X_{s})d(A^{-u^{*}}_{s}-N_s^{|u|_E})\nonumber\\
& &   +\int^{t}_{0}e^{{M^{-u^{*}}_{s-}-N_{s-}^{|u|_E}}}f(X_{s-})d({M^{-u^{*}}_{s}-N_s^{|u|_E}})\nonumber\\
&&+\frac{1}{2}\int^{t}_{0}e^{{M^{-u^{*}}_{s-}-N_{s-}^{|u|_E}}}f(X_{s-})d<M^{-u^{*},c}>_{s}+\int^{t}_{0}e^{{M^{-u^{*}}_{s-}-N_{s-}^{|u|_E}}}d<M^{f,c},M^{-u^{*},c}>_{s}\nonumber\\
&&+\sum_{s\leq
t}[f(X_{s})e^{{M^{-u^{*}}_{s}-N_s^{|u|_E}}}-f(X_{s-})e^{{M^{-u^{*}}_{s-}-N_{s-}^{|u|_E}}}\nonumber\\
& &\ \ \ \ \ \ -e^{{M^{-u^{*}}_{s-}-N_{s-}^{|u|_E}}}\triangle
f(X_{s})-f(X_{s-})e^{{M^{-u^{*}}_{s-}-N_{s-}^{|u|_E}}}\triangle
M^{-u^{*}}_{s}]\nonumber\\
&=&\left\{f(X_{0})+\int^{t}_{0}e^{{M^{-u^{*}}_{s-}-N_{s-}^{|u|_E}}}dM^{f}_{s}+\int^{t}_{0}e^{{M^{-u^{*}}_{s-}-N_{s-}^{|u|_E}}}L^{\bar{Q}^{u,n}}f(X_{s})ds\right.\nonumber\\
& &\ \ \ \ \ \ \left.
    +\int^{t}_{0}e^{{M^{-u^{*}}_{s-}-N_{s-}^{|u|_E}}}f(X_{s-})d{M^{-u^{*}}_{s}}\right\}\nonumber\\
&&+\left\{-\int_{0}^{t}e^{{M^{-u^{*}}_{s-}-N_{s-}^{|u|_E}}}f(X_{s})dA^{-u^{*}}_{s}+\frac{1}{2}\int^{t}_{0}e^{{M^{-u^{*}}_{s-}-N_{s-}^{|u|_E}}}f(X_{s-})d<M^{-u^{*},c}>_{s}\right.\nonumber\\
&&\left.-\int^{t}_{0}e^{{M^{-u^{*}}_{s-}-N_{s-}^{|u|_E}}}d<M^{f},M>_{s}+\int^{t}_{0}e^{{M^{-u^{*}}_{s-}-N_{s-}^{|u|_E}}}d<M^{f,c},M^{-u^{*},c}>_{s}\right.\nonumber\\
&&\left.+\sum_{s\leq
t}[f(X_{s})e^{{M^{-u^{*}}_{s}-N_s^{|u|_E}}}-f(X_{s-})e^{{M^{-u^{*}}_{s-}-N_{s-}^{|u|_E}}}\right.\nonumber\\
& &\left.\ \ \ \ \ \
-e^{{M^{-u^{*}}_{s-}-N_{s-}^{|u|_E}}}\triangle
f(X_{s})-f(X_{s-})e^{{M^{-u^{*}}_{s-}-N_{s-}^{|u|_E}}}\triangle
M^{-u^{*}}_{s}]\right\}\nonumber\\
&=&\left\{f(X_{0})+\int^{t}_{0}e^{{M^{-u^{*}}_{s-}-N_{s-}^{|u|_E}}}dM^{f}_{s}+\int^{t}_{0}e^{{M^{-u^{*}}_{s-}-N_{s-}^{|u|_E}}}L^{\bar{Q}^{u,n}}f(X_{s})ds\right.\nonumber\\
& &\ \ \ \ \ \ \left.
    +\int^{t}_{0}e^{{M^{-u^{*}}_{s-}-N_{s-}^{|u|_E}}}f(X_{s-})d{M^{-u^{*}}_{s}}\right\}\nonumber\\
&&+\left\{-\int_{0}^{t}e^{{M^{-u^{*}}_{s-}-N_{s-}^{|u|_E}}}f(X_{s-})dB^{p}_{s}-\int^{t}_{0}e^{{M^{-u^{*}}_{s-}-N_{s-}^{|u|_E}}}d<M^{f,d},M^{d}>_{s}\right.\nonumber\\
    &&\left. -\int^{t}_{0}e^{{M^{-u^{*}}_{s-}-N_{s-}^{|u|_E}}}d<M^{f,d},M^{-u^{*},d}>_{s}\right.\nonumber\\
    & &\left.+\sum_{s\leq
t}[f(X_{s})e^{{M^{-u^{*}}_{s}-N_s^{|u|_E}}}-f(X_{s-})e^{{M^{-u^{*}}_{s-}-N_{s-}^{|u|_E}}}\right.\nonumber\\
& &\left.\ \ \ \ \ \
-e^{{M^{-u^{*}}_{s-}-N_{s-}^{|u|_E}}}\triangle
f(X_{s})-f(X_{s-})e^{{M^{-u^{*}}_{s-}-N_{s-}^{|u|_E}}}\triangle
M^{-u^{*}}_{s}]\right\}\nonumber\\
&:=&I+II.
\end{eqnarray}
Note that
\begin{eqnarray}\label{add4}
II&=&-\int^{t}_{0}e^{{M^{-u^{*}}_{s-}-N_{s-}^{|u|_E}}}f(X_{s-})dB^{p}_{s}+\sum_{s\leq
      t}[-e^{{M^{-u^{*}}_{s-}-N_{s-}^{|u|_E}}}\triangle f(X_{s})\triangle B_{s}
       \nonumber\\
       & &-e^{{M^{-u^{*}}_{s-}-N_{s-}^{|u|_E}}}\triangle f(X_{s})\triangle M^{-u^{*}}_{s}]+\sum_{s\leq t}[f(X_{s})e^{M^{-u}_{s}-N_{s}^{|u|_E}}\nonumber\\
&&-f(X_{s-})e^{{M^{-u^{*}}_{s-}-N_{s-}^{|u|_E}}}
-e^{{M^{-u^{*}}_{s-}-N_{s-}^{|u|_E}}}\triangle f(X_{s})-f(X_{s-})e^{{M^{-u^{*}}_{s-}-N_{s-}^{|u|_E}}}\triangle M^{-u^{*}}_{s}]\nonumber\\
&=&-\int^{t}_{0}e^{{M^{-u^{*}}_{s-}-N_{s-}^{|u|_E}}}f(X_{s-})dB^{p}_{s}+\sum_{s\leq
      t}[-e^{{M^{-u^{*}}_{s-}-N_{s-}^{|u|_E}}}\triangle f(X_{s})(e^{\triangle
   M^{-u^{*}}_{s}}-1)\nonumber\\
&&+f(X_{s})e^{M^{-u^{*}}_{s}-N_{s}^{|u|_E}}-f(X_{s-})e^{M^{-u^{*}}_{s-}-N_{s-}^{|u|_E}}
-e^{M^{-u^{*}}_{s-}-N_{s-}^{|u|_E}}\triangle
f(X_{s}))\nonumber\\
& &-f(X_{s-})e^{M^{-u^{*}}_{s-}-N_{s-}^{|u|_E}}\triangle
M^{-u^{*}}_{s}]\nonumber\\
&=&-\int^{t}_{0}e^{{M^{-u^{*}}_{s-}-N_{s-}^{|u|_E}}}f(X_{s-})dB^{p}_{s}+\sum_{s\leq
  t}[e^{M^{-u^{*}}_{s-}-N_{s-}^{|u|_E}}f(X_{s-})e^{\triangle
   M^{-u^{*}}_{s}}\nonumber\\
&&-f(X_{s-})e^{M^{-u^{*}}_{s-}-N_{s-}^{|u|_E}}-f(X_{s-})e^{M^{-u^{*}}_{s-}-N_{s-}^{|u|_E}}\triangle M^{-u^{*}}_{s}]\nonumber\\
 &=&-\int^{t}_{0}e^{M^{-u^{*}}_{s-}-N_{s-}^{|u|_E}}f(X_{s-})dB^{p}_{s}+\int^{t}_{0}e^{M^{-u^{*}}_{s-}-N_{s-}^{|u|_E}}f(X_{s-})dB_{s}\nonumber\\
&=&\int_{0}^{t}e^{M^{-u^{*}}_{s-}-N_{s-}^{|u|_E}}f(X_{s-})dM^{d}_{s}.
\end{eqnarray}
Therefore (\ref{add2}) follows from (\ref{add3}) and (\ref{add4}).
 \end{proof}

\subsection [short title]{The bilinear form associated with
${(\bar{P}^{u,n}_{t})_{t\geq0}}$ on ${L^2(E_n;e^{-2u^{*}}m)}$}

For $n\in\mathbf{N}$, since $u^{*}\cdot I_{E_n}$ is bounded,
$(\bar{P}^{u,n}_{t})_{t\geq0}$ is also a strongly continuous
semigroup on $L^{2}(E_n;e^{-2u^{*}}m)$ by Theorem \ref{th2.1}. In
this subsection, we will study the bilinear form associated with
${(\bar{P}^{u,n}_{t})_{t\geq0}}$ on ${L^2(E_n;e^{-2u^{*}}m)}$.

Define $D(\mathcal{E})_{n,b}:=D(\mathcal{E})_n\cap
L^{\infty}(E;m)$.
 Let $f,g\in D(\mathcal{E})_{n,b}$. Note that
$e^{-2u^{*}}g=(e^{-2u^{*}}-1)g+g\in D(\mathcal{E})_{n,b}$. Define
\begin{equation}\label{add11}
\mathcal{E}^{u,n}(f,g):=\bar{Q}^{u,n}(f,e^{-2u^{*}}g),\ \ f,g\in
D(\mathcal{E})_{n,b}.
\end{equation}
Then, by Theorem \ref{th2.1}, we get
\begin{equation}\label{dog1}
 \mathcal{E}^{u,n}(f,g)=\lim_{t\rightarrow0}\frac{1}{t}(f-\bar{P}^{u,n}_{t}f,e^{-2u^{*}}g)_m=
 \lim_{t\rightarrow0}\frac{1}{t}(f-\bar{P}^{u.n}_{t}f,g)_{e^{-2u^{*}}m}.
\end{equation}
$(\mathcal{E}^{u,n},D(\mathcal{E})_{n,b})$ is called the bilinear
from associated with $(\bar{P}^{u,n}_{t})_{t\geq0}$ on
$L^2(E_n;e^{-2u^{*}}m)$.

Note that
\begin{eqnarray*}
<M^{f}\hskip -0.5cm&, & \hskip -0.4cm M^{d}>_{t}\nonumber\\
&=&[M^{f}, M^{d}]_{t}^{p}\nonumber\\
   &=&\left\{\sum_{s\leq
   t}[f(X_{s})-f(X_{s-})][e^{(u^{*}(X_{s-})-u^{*}(X_{s}))}-1-(u^{*}(X_{s-})-u^{*}(X_{s}))]\right\}^{p}\nonumber\\
   &=&\int_{0}^{t}\int_{E_{\triangle}}[f(y)-f(X_{s})]
        [e^{(u^{*}(X_{s})-u^{*}(y))}-1-(u^{*}(X_{s})-u^{*}(y))]N(X_{s},dy)dH_{s}.
\end{eqnarray*}
Then
\begin{eqnarray}\label{add41}
\int_Eg\hskip -0.5cm&& \hskip -0.5cm d\mu_{<M^{f},M^{d}>}\nonumber\\
&=&\int_E\int_{E_{\triangle}}g(x)[f(y)-f(x)]
        [e^{(u^{*}(x)-u^{*}(y))}-1-(u^{*}(x)-u^{*}(y))]N(x,dy)\nu(dx).\nonumber\\
        & &
\end{eqnarray}
By (\ref{add34}) and (\ref{ap1}), we get
\begin{eqnarray}\label{add42}
\int_Efgd\mu_{-u}&=&\int_E\int_{E_{\triangle}}
        f(x)g(x)[e^{(u^{*}(x)-u^{*}(y))}-1-(u^{*}(x)-u^{*}(y))]N(x,dy)\nu(dx)\nonumber\\
        & &+\frac{1}{2}\int_Efgd\mu_{<M^{-u^{*},c}>}+\int_Efgd\eta_u-\int_Efg|u|_Edm.
\end{eqnarray}

Similar to \cite[Theorem 5.3.1]{Fu94} (cf. also \cite[Chapter
5]{oshima}), we can show that $J(dx,dy)=\frac{1}{2}N(y,dx)\nu(dy)$
and $K(dx)=N(x,\triangle)\nu(dx)$.  Therefore, we obtain by
(\ref{add11}), (\ref{e2.8}), (\ref{add41}) and (\ref{add42}) that
\begin{eqnarray}\label{e2.14}
\mathcal{E}^{u,n}(f,g)&=&\bar{Q}^{u.n}(f,e^{-2u^{*}}g)\nonumber\\
\nonumber&=&\mathcal{E}(f,e^{-2u^{*}}g)-\int_{E}e^{-2u^{*}}gd\mu_{<M^{f},M>}-\int_{E}e^{-2u^{*}}fgd\mu_{-u}\\
\nonumber&=&\mathcal{E}(f,e^{-2u^{*}}g)-\int_{E}e^{-2u^{*}}gd\mu_{<M^{f},M^{-u^{*}}>}-\int_{E}e^{-2u^{*}}gd\mu_{<M^{f},M^{d}>}\\
& &-\int_{E}e^{-2u^{*}}fgd\mu_{-u}\nonumber\\
\nonumber
&=&\mathcal{E}(f,e^{-2u^{*}}g)-\int_{E}e^{-2u^{*}}gd\mu_{<M^{f},M^{-u^{*}}>}\nonumber\\
&&-2\int_{E\times
E-d}e^{-2u^{*}(y)}g(y)f(x)[e^{(u^{*}(y)-u^{*}(x))}-1-(u^{*}(y)-u^{*}(x))]J(dx,dy)\nonumber\\
 & &       -\frac{1}{2}\int_{E}
        e^{-2u^{*}}fgd\mu_{<M^{-u^{*},c}>}-{\cal E}(|u|_E,e^{-2u^{*}}fg).\nonumber\\
        & &
\end{eqnarray}

\begin{thm}\label{th0} For each $n\in\mathbf{N}$, under either the assumption of
Theorem \ref{th1.1} or Theorem \ref{th1.2}, we have
\begin{equation}\label{new2}
{\cal E}^{u,n}(f,g)={Q}^{u}(fe^{-u^{*}},ge^{-u^{*}}),\ \ f,g\in
D(\mathcal{E})_{n,b}.
\end{equation}
\end{thm}

\begin{proof} We fix an $n\in\mathbf{N}$. Define
\begin{eqnarray}\label{ap3}
\Psi^{u^{*},n}(f\hskip-0.4cm&,&\hskip-0.3cm g):=\mathcal{E}(f,e^{-2u^{*}}g)-\int_{E}e^{-2u^{*}}gd\mu_{<M^{f},M^{-u^{*}}>}\nonumber\\
&&-2\int_{E\times
E-d}e^{-2u^{*}(y)}g(y)f(x)[e^{(u^{*}(y)-u^{*}(x))}-1-(u^{*}(y)-u^{*}(x))]J(dx,dy)\nonumber\\
 & &       -\frac{1}{2}\int_{E}
        e^{-2u^{*}}fgd\mu_{<M^{-u^{*},c}>},\ \ f,g\in D(\mathcal{E})_{n,b}.
\end{eqnarray}
Then, by (\ref{e2.14}) and (\ref{ap101}), we find that
(\ref{new2}) is equivalent to
\begin{equation}\label{ap2}
\Psi^{u^{*},n}(f,g)={\cal E}(fe^{-u^{*}},ge^{-u^{*}})+{\cal
E}(u^{*},e^{-2u^{*}}fg),\ \ f,g\in D(\mathcal{E})_{n,b}.
\end{equation}

Since $u^{*}\cdot I_{E_n}$ is bounded, there exists
$l_0\in\mathbf{N}$ such that $|u^{*}(x)|\le l_0$ for all $x\in
E_n$. For $l\in \mathbf{N}$, define $u^{*}_l:=((-l)\vee
u^{*})\wedge l$. Then $u^{*}_l\in D({\cal E})_b$ and
$u^{*}=u^{*}_l$ on $E_n$ for $l\ge l_0$. Similar to \cite[Lemma
5.3.1]{Fu94}, we can show that
$\mu_{<M^{-u^{*},c}>}|_{E_n}=\mu_{<M^{-u^{*}_l,c}>}|_{E_n}$ for
$l\ge l_0$. For $\phi\in D({\cal E})_b$, we define
\begin{eqnarray}\label{ap102}
\Psi^{\phi,n}(f\hskip-0.4cm&,&\hskip-0.3cm g):=\mathcal{E}(f,e^{-2\phi}g)-\int_{E}e^{-2\phi}gd\mu_{<M^{f},M^{-\phi}>}\nonumber\\
&&-2\int_{E\times
E-d}e^{-2\phi(y)}g(y)f(x)[e^{(\phi(y)-\phi(x))}-1-(\phi(y)-\phi(x))]J(dx,dy)\nonumber\\
 & &       -\frac{1}{2}\int_{E}
        e^{-2\phi}fgd\mu_{<M^{-\phi,c}>},\ \ f,g\in D(\mathcal{E})_{n,b}.
\end{eqnarray}
Then, by (\ref{ap3}) and (\ref{ap102}), we find that for $l\ge
l_0$
$$
\Psi^{u^{*},n}(f,g)=\Psi^{u^{*}_l,n}(f,g)+\int_{E}e^{-2u^{*}}gd\mu_{<M^{f},M^{u^{*}-u^{*}_l}>},\
\ f,g\in D(\mathcal{E})_{n,b}.
$$

Note that by \cite[(5.1.3)]{oshima}
\begin{eqnarray*}
\left|\int_{E}e^{-2u^{*}}gd\mu_{<M^{f},M^{u^{*}-u^{*}_l}>}\right|&\le&
2e^{2l_0}\|g\|_{\infty}{\cal E}(f,f)^{\frac{1}{2}}{\cal E}(u^{*}-u^{*}_l,u^{*}-u^{*}_l)^{\frac{1}{2}}\\
&\rightarrow&0\ \ {\rm as}\ l\rightarrow \infty,
\end{eqnarray*}
and $${\cal E}(u^{*}_l,e^{-2u^{*}}fg)\rightarrow{\cal
E}(u^{*},e^{-2u^{*}}fg)\ \ {\rm as}\ l\rightarrow \infty.
$$
Hence, to establish (\ref{ap2}), it is sufficient to show that for
any $\phi\in D({\cal E})_{b}$ and $f,g\in D(\mathcal{E})_{n,b}$
\begin{equation}\label{ap103}
\Psi^{\phi,n}(f,g)={\cal E}(fe^{-\phi},ge^{-\phi})+{\cal
E}(\phi,e^{-2\phi}fg).
\end{equation}

Let $\phi\in D({\cal E})_{b}$. By \cite[(5.3.2)]{oshima}, we have
\begin{equation}\label{new1}\int gd\mu_{<M^{f},M^{-\phi}>}
=-\mathcal{E}(f,g\phi)-\mathcal{E}(\phi,gf)+\mathcal{E}(f\phi,g).
\end{equation}
By (\ref{ap102}) and (\ref{new1}), we find that (\ref{ap103}) is
equivalent to
\begin{eqnarray}\label{jack1}
\mathcal{E}(f\hskip -0.5cm&, & \hskip -0.4cm e^{-2\phi}g)+\mathcal{E}(f,e^{-2\phi}g\phi)-\mathcal{E}(f\phi,e^{-2\phi}g)\nonumber\\
& &-2\int_{E\times
E-d}e^{-2\phi(y)}g(y)f(x)[e^{(\phi(y)-\phi(x))}-1-(\phi(y)-\phi(x))]J(dx,dy)\nonumber\\
& &-\frac{1}{2}\int_{E}
        e^{-2\phi}fgd\mu_{<M^{-\phi,c}>}\nonumber\\
        &=&{\cal E}(fe^{-\phi},ge^{-\phi}).
\end{eqnarray}

Denote by $M^{-\phi,j}_{t}$ and $M^{-\phi,k}_{t}$ the jumping and
killing parts of $M^{-\phi}_{t}$, respectively. Then, similar to
\cite[(5.3.9) and (5.3.10)]{Fu94}, we get
$$
\mu_{<M^{-\phi,j}>}(dx)=2\int_E(\phi(x)-\phi(y))^2J(dy,dx)\ \ {\rm
and}\ \ \mu_{<M^{-\phi,k}>}(dx)=\phi^2(x)K(dx).
$$
Thus, for any $w\in D({\cal E})_b$, we have
\begin{eqnarray}\label{star1}
\int_Ewd\mu_{<M^{-\phi,c}>}&=&\int_Ewd(\mu_{<M^{-\phi}>}-\mu_{<M^{-\phi,j}>}-\mu_{<M^{-\phi,k}>})\nonumber\\
&=&2{\cal E}(\phi,\phi w)-{\cal E}(\phi^2,w)\nonumber\\
& &-2\int_{E\times E-d}(\phi(y)-\phi(x))^2w(y)J(dx,dy)-\int_Ew\phi^2dK.\nonumber\\
& &
\end{eqnarray}

By (\ref{star1}), we find that (\ref{jack1}) is equivalent to
\begin{eqnarray}\label{newnew}
\mathcal{E}(f\hskip -0.5cm&, & \hskip -0.4cm
e^{-2\phi}g)+\mathcal{E}(f,e^{-2\phi}g\phi)-\mathcal{E}(f\phi,e^{-2\phi}g)-{\cal
E}(\phi,e^{-2\phi}\phi fg)
+\frac{1}{2}{\cal E}(\phi^2,e^{-2\phi}fg)\nonumber\\
& &-2\int_{E\times
E-d}e^{-2\phi(y)}g(y)f(x)[e^{(\phi(y)-\phi(x))}-1-(\phi(y)-\phi(x))]J(dx,dy)\nonumber\\
& &+\int_{E\times
E-d}(\phi(y)-\phi(x))^2e^{-2\phi(y)}f(y)g(y)J(dx,dy)+\frac{1}{2}\int_Ee^{-2\phi}fg\phi^2dK\nonumber\\
        &=&{\cal E}(fe^{-\phi},ge^{-\phi}).
\end{eqnarray}

\noindent {\bf Proof of (\ref{newnew}) under the assumption of
Theorem \ref{th1.1}.}

\noindent Denote by $\tilde{\cal E}$ the symmetric part of ${\cal
E}$. Then $(\tilde{\cal E}, D({\cal E}))$ is a regular symmetric
Dirichlet form. Denote by $\tilde{J}$ and $\tilde{K}$ the jumping
and killing measures of $(\tilde{\mathcal{E}},D(\mathcal{E}))$,
respectively. Then
\begin{eqnarray}\label{star11}
\int_{E\times
                 E-d}\hskip -0.5cm& & \hskip -0.5cm (\phi(y)-\phi(x))^2J(dx,dy)+\int_E\phi^2dK\nonumber\\
                 &\le&2\left\{\int_{E\times
                 E-d}(\phi(y)-\phi(x))^2\tilde{J}(dx,dy)+\int_E\phi^2d\tilde{K}\right\}\nonumber\\
                 &\le& 2{\cal
                 E}(\phi,\phi)
\end{eqnarray}
and
\begin{eqnarray}\label{star12}
\int_{E\times
E-d}\hskip -0.5cm& & \hskip -0.4cm [e^{(\phi(y)-\phi(x))}-1-(\phi(y)-\phi(x))]J(dx,dy)\nonumber\\
&\le&C_{\|\phi\|_{\infty}}\int_{E\times
    E-d}(\phi(y)-\phi(x))^{2}J(dx,dy)\nonumber\\
    &\leq& C_{\|\phi\|_{\infty}}\mathcal{E}(\phi,\phi)
    \end{eqnarray}
for some constant $C_{\|\phi\|_{\infty}}>0$.  Hence, to establish
(\ref{newnew}) for $\phi\in D({\cal E})_b$ and $f,g\in D({\cal
E})_{n,b}$, it is
  sufficient to  establish (\ref{newnew}) for  $\phi,f,g\in D:=C_0(E)\cap D({\cal
    E})$ by virtue of the density of $D$ in $D({\cal
    E})$ and approximation.

    By \cite[Theorem 4.8 and Proposition 5.1]{HC06}, we have
the following Beurling-Deny decomposition
\begin{eqnarray}\label{e2.4}
\mathcal{E}(f,g)&=&\mathcal{E}^{c}(f,g)+SPV\int_{E\times
E-d}2(f(y)-f(x))g(y)J(dx,dy)\nonumber\\
& &+\int_{E}fgdK,\ \ f,g\in D(\mathcal{E})_{b},
\end{eqnarray}
where $SPV\int$ denotes the symmetric principle value integral
(see \cite[Definition 2.5]{HC06}) and $\mathcal{E}^{c}(f,g)$
satisfies the left strong local property in the sense that
$\mathcal{E}^{c}(f,g)=0$ if $f$ is constant ${\cal E}$-q.e. on a
quasi-open set containing the quasi-support of $g$ (see
\cite[Theorem 4.1]{HC06}). By (\ref{e2.4}), we obtain that for any
$w\in D({\cal E})_b$,
\begin{eqnarray*}
2{\cal E}(\phi\hskip -0.5cm&, & \hskip -0.4cm \phi w)-{\cal E}(\phi^2,w)\\
& &-2\int_{E\times E-d}(\phi(y)-\phi(x))^2w(y)J(dx,dy)-\int_Ew\phi^2dK\\
&=&2{\cal E}^c(\phi,\phi w)-{\cal E}^c(\phi^2,w).
\end{eqnarray*}
Hence (\ref{newnew}) is equivalent to
\begin{eqnarray}\label{new4}
\mathcal{E}(f\hskip -0.5cm&, & \hskip -0.4cm
e^{-2\phi}g)+\mathcal{E}(f,e^{-2\phi}g\phi)-\mathcal{E}(f\phi,e^{-2\phi}g)
-{\mathcal{E}}^{c}(\phi,e^{-2\phi}\phi fg)+\frac{1}{2}{\mathcal{E}}^{c}(\phi^2,e^{-2\phi}fg)\nonumber\\
& &-2\int_{E\times
E-d}e^{-2\phi(y)}g(y)f(x)[e^{(\phi(y)-\phi(x))}-1-(\phi(y)-\phi(x))]J(dx,dy)\nonumber\\
        &=&{\cal E}(fe^{-\phi},ge^{-\phi}).
\end{eqnarray}

In the following, we will establish (\ref{new4}) by showing that
its left hand side and its right hand side have the same
diffusion, jumping and killing parts. We assume without loss of
generality that $\phi,f,g\in D$.

First, let us consider the diffusion parts of both sides of
(\ref{new4}). Following \cite[(3.4)]{HZC08}, we introduce a linear
functional $<L(w,v),\cdot>$ for $w,v\in D$ by
\begin{equation}\label{new11}
<L(w,v),f>:=\check{\mathcal{E}}^c(w,vf):=\frac{1}{2}({\cal
E}^c(w,vf)-\hat{\cal E}^c(w,vf)),\ \ f\in D,
\end{equation}
where $\hat{\cal E}^c$ is the left strong local part of the dual
Dirichlet form $(\hat{\cal E}, D({\cal E}))$. Define
\begin{eqnarray*}
D_{\rm loc}:=\{w|\mbox{ for any relatively compact open
set } G\ {\rm of}\ E, {\rm there} \nonumber\\
\mbox{ {\rm \ exists} a function}\ v\in D \mbox{ such that } w=v
\mbox{ on } G\}.
\end{eqnarray*}
Then, the linear functional $<L(w,v),\cdot>$ can be extended and
defined for any $w,v\in D_{\rm loc}$ (cf. \cite[Definition
3.6]{HZC08}). Note that $J_1$ is assumed to be finite. Similar to
\cite[Theorem 3.8]{HZC08}, we can prove the following lemma.
\begin{lem}
Let $w_1,\dots,w_l,v\in D_{\rm loc}$ and $f\in D$. Denote
$w:=(w_1,\dots,w_l)$. If $\psi\in C^2({\mathbf R^l})$, then
$\psi(w)\in D_{\rm loc}$, $\psi_{x_i}(w)\in D_{\rm loc}$, $1\leq
i\leq l$, and
\begin{eqnarray}\label{pro3.10-1}
<L(\psi(w),v),f>=\sum_{i=1}^{l}<L(w_i,v),\psi_{x_i}(w)f>.
\end{eqnarray}
\end{lem}

By (\ref{new11}) and (\ref{pro3.10-1}), we get
\begin{eqnarray}\label{add01}
{\check{\mathcal{E}}}^{c}(f\hskip-0.5cm&,&\hskip-0.4cm e^{-2\phi}g)+{\check{\mathcal{E}}}^{c}(f,e^{-2\phi}g\phi)-{\check{\mathcal{E}}}^{c}(f\phi,e^{-2\phi}g)\nonumber\\
& &-\check{\mathcal{E}}^{c}(\phi,e^{-2\phi}\phi fg)+\frac{1}{2}{\check{\mathcal{E}}}^{c}(\phi^2,e^{-2\phi}fg)\nonumber\\
&=&{\check{\mathcal{E}}}^{c}(f,e^{-2\phi}g)+{\check{\mathcal{E}}}^{c}(f,e^{-2\phi}g\phi)-{\check{\mathcal{E}}}^{c}(f\phi,e^{-2\phi}g)\nonumber\\
&=&{\check{\mathcal{E}}}^{c}(f,e^{-2\phi}g)-{\check{\mathcal{E}}}^{c}(\phi,e^{-2\phi}fg)\nonumber\\
&=&{\check{\mathcal{E}}}^{c}(f,e^{-2\phi}g)+{\check{\mathcal{E}}}^{c}(e^{-\phi},e^{-\phi}fg)\nonumber\\
&=&{\check{\mathcal{E}}}^{c}(fe^{-\phi},ge^{-\phi}).
\end{eqnarray}
By LeJan's formula (cf.  \cite[Theorem 3.2.2 and Page 117]{Fu94},
we can check that
\begin{eqnarray}\label{add02}
\tilde{\mathcal{E}}^{c}(f\hskip-0.5cm&,&\hskip-0.4cm
e^{-2\phi}g)+\tilde{\mathcal{E}}^{c}(f,e^{-2\phi}g\phi)
             -\tilde{\mathcal{E}}^{c}(f\phi,e^{-2\phi}g)\nonumber\\
             & &-\tilde{\mathcal{E}}^{c}(\phi,e^{-2\phi}\phi fg)+\frac{1}{2}\tilde{\mathcal{E}}^{c}(\phi^2,e^{-2\phi}fg)\nonumber\\
&=&\frac{1}{2}\int_{E}d\tilde\mu^{c}_{<f,e^{-2\phi}g>}+\frac{1}{2}\int_{E}d\tilde\mu^{c}_{<f,e^{-2\phi}g\phi>}
  -\frac{1}{2}\int_{E}d\tilde\mu^{c}_{<f\phi,e^{-2\phi}g>}\nonumber\\
&&-\frac{1}{2}\int_{E}d\tilde\mu^{c}_{<\phi,e^{-2\phi}\phi fg>}+\frac{1}{4}\int_{E}d\tilde\mu^{c}_{<\phi^2,e^{-2\phi}fg>}\nonumber\\
&=&\frac{1}{2}\int_{E}d\tilde\mu^{c}_{<fe^{-\phi},ge^{-\phi}>}\nonumber\\
&=&\tilde{\mathcal{E}}^{c}(fe^{-\phi},ge^{-\phi}),
\end{eqnarray}
where $\tilde{\mathcal{E}}^{c}$ denotes the strong local part of
 $(\tilde{\cal E}, D({\cal E}))$ and
$\tilde\mu^c$ denotes the local part of energy measure w.r.t.
$(\tilde{\cal E}, D({\cal E}))$. Then the diffusion parts of both
sides of (\ref{new4}) are equal by (\ref{add01}) and
(\ref{add02}).

For the jumping parts of (\ref{new4}), we have
\begin{eqnarray*}
\mathcal{E}^{j}(f\hskip-0.5cm&,&\hskip-0.4cm e^{-2\phi}g)+\mathcal{E}^{j}(f,e^{-2\phi}g\phi)-\mathcal{E}^{j}(f\phi,e^{-2\phi}g)-\mathcal{E}^{j}({fe^{-\phi},ge^{-\phi}})\\
&&-2\int_{E\times
  E-d}e^{-2\phi(y)}g(y)f(x)[e^{(\phi(y)-\phi(x))}-1
  -(\phi(y)-\phi(x))]J(dx,dy)\\
&=&2SPV\int_{E\times E-d}\{(f(y)-f(x))e^{-2\phi(y)}g(y)+(f(y)-f(x))\phi(y)e^{-2\phi(y)}g(y)\\
&&-(f(y)\phi(y)-f(x)\phi(x))e^{-2\phi(y)}g(y)-(f(y)e^{-\phi(y)}-f(x)e^{-\phi(x)})e^{-\phi(y)}g(y)\\
&&-e^{-2\phi(y)}g(y)f(x)[e^{(\phi(y)-\phi(x))}-1-(\phi(y)-\phi(x))]\}J(dx,dy)\\
&=&0.
\end{eqnarray*}
For the killing parts of (\ref{new4}), we have
\begin{eqnarray*}
\mathcal{E}^{k}(f\hskip-0.5cm&,&\hskip-0.4cm e^{-2\phi}g)+\mathcal{E}^{k}(f,e^{-2\phi}g\phi)-\mathcal{E}^{k}(f\phi,e^{-2\phi}g)-\mathcal{E}^{k}(fe^{-\phi},ge^{-\phi})\\
&=&\int_{E}(fe^{-2\phi}g+fe^{-2\phi}g\phi-f\phi e^{-2\phi}g-fe^{-2\phi}g)dK\\
&=&0.
\end{eqnarray*}
The proof is complete.

\vskip 0.4cm \noindent {\bf Proof of (\ref{newnew}) under the
assumption of Theorem \ref{th1.2}.}

\noindent Let $G$ be a relatively compact open subset of $U$ such
that the distance between the boundary of $G$ and that of $U$ is
greater than some constant $\delta>0$. Then, similar to
\cite[Theorem 4.8]{HZC08}, we can show that $({\cal E},
C^{\infty}_0(G))$ has the following representation:
\begin{eqnarray}\label{star44}
{\mathcal{E}}(w,v)&=&\sum_{i,j=1}^d\int_U\frac{\partial
w}{\partial x_i}\frac{\partial v}{\partial x_j}d\nu^G_{ij}
+\sum_{i=1}^d<F^G_i,\frac{\partial
w}{\partial x_i}v>\nonumber\\
& &+\ SPV\int_{U\times U-d}
2\left(\sum_{i=1}^d(y_i-x_i)\frac{\partial w}{\partial
y_i}(y)I_{\{|x-y|\leq
\frac{\delta}{2}\}}(x,y)\right)v(y)\tilde{J}(dx,dy)\nonumber\\
 &
&+\int_{U\times U-
d}2\left(w(y)-w(x)-\sum_{i=1}^d(y_i-x_i)\frac{\partial w}{\partial
y_i}(y)I_{\{|x-y|\leq \frac{\delta}{2}\}}(x,y)\right)v(y)J(dx,dy)\nonumber\\
& &+\int_U wvdK,\ \ \ \ \ \ w,v\in C^{\infty}_0(G),
\end{eqnarray}
where $\{\nu^G_{ij}\}_{1\le i,j\le d}$ are signed Radon measures
on $U$ such that for every $K\subset U$, $K$ is compact,
$\nu^G_{ij}(K)=\nu^G_{ji}(K)$ and
$\sum_{i,j=1}^d\xi_i\xi_j\nu^G_{ij}(K)\ge 0$ for all
$\xi=(\xi_1,\dots,\xi_d)\in \mathbf{R}^{d}$, $\{F^G_i\}_{1\leq
i\leq d}$ are generalized functions  on $U$.

By (\ref{star44}), we can check that (\ref{newnew}) holds for all
$\phi,f,g\in C^{\infty}_0(U)$. Therefore (\ref{newnew}) holds for
$\phi\in D({\cal E})_b$ and $f,g\in D({\cal E})_{n,b}$ by
(\ref{star11}), (\ref{star12}) and approximation. The proof is
complete.
\end{proof}

\subsection [short title]{Proof of Theorems \ref{th1.1} and \ref{th1.2} and some remarks}

\begin{proof} By Theorem \ref{th2.1}, for each $n\in\mathbf{N}$,
$(\bar{P}^{u,n}_{t})_{t\geq0}$ is a strongly continuous semigroup
of bounded operators on $L^{2}(E_n;m)$ with
$\|\bar{P}^{u,n}_{t}\|_{2}\leq e^{\beta_{n}t}$ for every $t>0$ and
some constant $\beta_n>0$. Moreover, the coercive closed form
associated with $(e^{-\beta_nt}\bar{P}^{u,n}_{t})_{t\geq0}$ is
given by $(\bar{Q}^{u,n}_{\beta_n},D(\mathcal{E})_n)$. Note that
$(\bar{P}^{u,n}_{t})_{t\geq0}$ is also a strongly continuous
semigroup of bounded operators on $L^{2}(E_n;e^{-2u^{*}}m)$ and
the bilinear from associated with $(\bar{P}^{u,n}_{t})_{t\geq0}$
on $L^2(E_n;e^{-2u^{*}}m)$ is given by
$(\mathcal{E}^{u,n},D(\mathcal{E})_{n,b})$ (see (\ref{dog1})).

Define
$$P^{u,n}_{t}f(x):=E_{x}[e^{N^{u}_{t}}f(X_{t});t<\tau_{E_n}].$$
Then
\begin{eqnarray}\label{dog2}
P^{u,n}_{t}f(x)&=&E_{x}[e^{N^{u^{*}}_{t}-N^{|u|_E}_t}f(X_{t});t<\tau_{E_n}]\nonumber\\
&=&E_{x}[e^{u^{*}(X_{t})-u^{*}(X_{0})+M^{-u^{*}}_{t}-N_t^{|u|_E}}f(X_{t});t<\tau_{E_n}]\nonumber\\
&=&e^{-u^{*}(x)}\bar{P}^{u,n}_{t}(e^{u^{*}}f)(x).\end{eqnarray}
Hence $({P}^{u,n}_{t})_{t\geq0}$ is a strongly continuous
semigroup of bounded operators on $L^{2}(E_n;m)$. Let
$(Q^{u,n},D(\mathcal{E})_{n,b})$  be the restriction of $Q^{u}$ to
$D(\mathcal{E})_{n,b}$. Then, by (\ref{dog2}), (\ref{dog1}) and
Theorem \ref{th0}, we know that the bilinear from associated with
$({P}^{u,n}_{t})_{t\geq0}$ on $L^2(E_n;m)$ is given by
$(Q^{u,n},D(\mathcal{E})_{n,b})$. That is,
\begin{equation}\label{dog11}
Q^{u,n}(f,g)=
 \lim_{t\rightarrow0}\frac{1}{t}(f-{P}^{u,n}_{t}f,g)_{m},\ \ f,g\in D(\mathcal{E})_{n,b}.
\end{equation}

\noindent (i) Suppose that there exists a constant
$\alpha_{0}\geq0$ such that
$$
  Q^{u}(f,f)\geq -\alpha_0(f,f)_m,\ \ \forall f\in
D(\mathcal{E})_{b}.
$$
For $n\in\mathbf{N}$, let $(L^{n},D(L^{n}))$ be the generator of
$(P^{u,n}_{t})_{t\geq0}$ on $L^2(E_n;m)$. Then
$D(L^{n}-\alpha_{0})$ is dense in $L^{2}(E_{n};m)$.

Define \begin{equation}\label{dog3}
\bar{L}^nf(x)=e^{u^{*}(x)}L^n(e^{-u^{*}}f)(x),\ \ f\in
D(\bar{L}^n):=\{e^{u^{*}}g|g\in D(L^n)\}.
\end{equation}
Then, by (\ref{dog2}), $(\bar{L}^n,D(\bar{L}^n))$ is the generator
of $(\bar{P}^{u,n}_{t})_{t\geq0}$  on $L^{2}(E_n;e^{-2u^{*}}m)$.
$(\bar{L}^n,D(\bar{L}^n))$ is also the generator of
$(\bar{P}^{u,n}_{t})_{t\geq0}$ on $L^{2}(E_n;m)$ due to the
boundedness of $u^{*}$ on $E_n$. Since
$(e^{-\beta_nt}\bar{P}^{u,n}_{t})_{t\geq0}$ is a strongly
continuous contraction semigroup on $L^{2}(E_n;m)$, $ {\rm
Range}(\lambda-\bar{L}^{n})=L^{2}(E_{n};m)$ for all
$\lambda>\beta_n$. Hence
Range$(\lambda-({L}^{n}-\alpha_0))=L^{2}(E_{n};m)$ for all
$\lambda>\beta_n-\alpha_0$ by (\ref{dog3}).

Let $f\in L^{2}(E_{n};m)$. Then, for any $\alpha>0$, we obtain by
(\ref{dog11}) that
\begin{eqnarray*}
\nonumber\|[\alpha-(L^{n}-\alpha_{0})]f\|_2\cdot\|f\|_2&=&\|
           [(\alpha+\alpha_{0})-L^{n}]f\|_2\cdot\|f\|_2\\
\nonumber&\geq&
             ([(\alpha+\alpha_{0})-L^{n}]f,f)_{m}\\
\nonumber&=&Q^{u,n}(f,f)+(\alpha+\alpha_0)(f,f)_{m}\\
&\geq&\alpha(f,f)_{m}.
\end{eqnarray*}
Hence $L^{n}-\alpha_{0}$ is dissipative on $L^{2}(E_{n};m)$.
Therefore $(e^{-\alpha_{0}t}P^{u,n}_{t})_{t\geq0}$ is a strongly
continuous contraction semigroup on $L^{2}(E_n;m)$ by the
Hille-Yosida theorem (cf. \cite[Chapter 1, Theorem 2.6]{EK}).

Let $g\in L^{2}(E;m)$ and $t>0$. Then
\begin{eqnarray*}
\|P^{u}_{t}g\|_2&\leq&\|P^{u}_{t}|g|\,\|_2\\
&=&\lim_{l\rightarrow\infty}\|P^{u}_{t}|g\cdot
I_{E_l}|\,\|_2\\
&\leq&\liminf_{l\rightarrow\infty}\liminf_{n\rightarrow\infty}\|P^{u,n}_{t}|g\cdot
I_{E_l}|\,\|_2\\
&\leq&e^{\alpha_{0}t}\|g\|_2.
\end{eqnarray*}
Since $g\in L^{2}(E;m)$ is arbitrary, we get $$
 \|P^{u}_{t}\|_2\leq
e^{\alpha_{0}t},\ \ \forall t>0. $$

\noindent (ii) Suppose that there exists a constant
$\alpha_{0}\geq0$ such that
\begin{eqnarray}\label{dog77}
\|P^{u}_{t}\|_2\leq e^{\alpha_{0}t},\ \ \forall t>0.
\end{eqnarray}
Let $n\in \mathbf{N}$ and $f\in L^{2}(E_{n};m)$. Then
$$
\|
P^{u,n}_{t}f\|_2\leq\|P^{u,n}_{t}|f|\,\|_2\leq\|P^{u}_{t}|f|\,\|_2\leq
e^{\alpha_{0}t}\|f\|_2.
$$
Hence $(e^{-\alpha_{0}t}{P}^{u,n}_{t})_{t\geq0}$ is a strongly
continuous contraction semigroup on $L^{2}(E_n;m)$. By
(\ref{dog11}), we get
\begin{equation}\label{dog101}
Q^{u,n}(f,f)+\alpha_0(f,f)_m =
\lim_{t\rightarrow0}\frac{1}{t}(f-e^{-\alpha_{0}t}{P}^{u,n}_{t}f,f)_{m}\ge
0,\ \ \forall f\in D(\mathcal{E})_{n,b}.
\end{equation}
By (\ref{dog101}) and approximation, we find that
$$
  Q^{u}(f,f)\geq -\alpha_0(f,f)_m,\ \ \forall f\in
D(\mathcal{E})_{b}.
$$

Now we show that $(P^{u}_{t})_{t\geq0}$ is strongly continuous on
$L^{2}(E;m)$. Let $n\in \mathbf{N}$ and $f\in L^{2}(E_{n};m)$
satisfying $f\ge 0$. Then, we obtain by (\ref{dog77}) and the
strong continuity of $({P}^{u,n}_{t})_{t\geq0}$ that
\begin{eqnarray*}
\limsup_{t\rightarrow 0}\hskip-0.48cm&
&\hskip-0.5cm\|f-e^{-\alpha_0
t}P^{u}_{t}f\|^2_2\\
&=&\limsup_{t\rightarrow 0}\{2(f-e^{-\alpha_0
t}P^{u}_{t}f,f)_m-[(f,f)_m-\|e^{-\alpha_0 t}P^{u}_{t}f\|^2_2]\}\\
&\le&2\limsup_{t\rightarrow 0}(f-e^{-\alpha_0 t}P^{u}_{t}f,f)_m\\
&\le&2\limsup_{t\rightarrow 0}(f-e^{-\alpha_0
t}P^{u,n}_{t}f,f)_m\\
&=&0.
\end{eqnarray*}
Since $f$ and $n$ are arbitrary, $({P}^{u}_{t})_{t\geq0}$ is
strongly continuous on $L^{2}(E;m)$ by (\ref{dog77}). The proof is
complete.
\end{proof}
\begin{rem}
Let $u\in D({\cal E}$). Define
\begin{equation}\label{ap22}
B^u_{t}=\sum_{s\leq
t}\left[e^{(u(X_{s-})-u(X_{s}))}-1-(u(X_{s-})-u(X_{s}))\right].
\end{equation}
Note that $(B^u_t)_{t\ge 0}$ may not be locally integrable (cf.
\cite[Theorem 3.3]{CSM07}). To overcome this difficulty, we
introduced the nonnegative function $u^{*}$ and the locally
integrable increasing process $(B_t)_{t\ge 0}$ (see (\ref{u}) and
(\ref{ap11})). This technique has been used in \cite{CSM07} to
show that if $X$ is symmetric and $u\in D({\cal E})_e$, then
$(P_{t}^{u})_{t\ge 0}$ is strongly continuous if and only if
$({{Q}}^{u},D(\mathcal{E})_{b})$ is lower semi-bounded. Here and
henceforth $D({\cal E})_e$ denotes the extended Dirichlet space of
$({\cal E},D({\cal E}))$.

In fact, if we assume that $({\cal E},D({\cal E}))$ satisfies the
strong sector condition instead of the weak sector condition (cf.
\cite[Pages 15 and 16]{MR92} for the definitions), then similar to
\cite[Page 158]{CSM07} we can introduce a function $|u|^g_E$ for
each $u\in D({\cal E})_e$. Define $u^{*}:=u+|u|^g_E$. Using this
defined $u^{*}$, similar to the above proof of this section, we
can show that Theorems 1.1 and 1.2 hold for all $u\in D({\cal
E})_e$.

On the other hand, suppose we still assume that $({\cal E},D({\cal
E}))$ satisfies the weak sector condition and $u\in D({\cal
E})_e$. Define
$$ F^u_{t}=\sum_{s\leq
t}\left[e^{(u(X_{s-})-u(X_{s}))}-1-(u(X_{s-})-u(X_{s}))\right]^2.
$$
If $(F^u_t)_{t\ge 0}$ is locally integrable on $[0, \zeta)$ for
q.e. $x\in E$, then we can show that Theorems 1.1 and 1.2 still
hold. The proof is similar to the above proof of this section but
we directly apply the $(B^u_t)_{t\ge 0}$ defined in (\ref{ap22})
instead of the $(B_t)_{t\ge 0}$ defined in (\ref{ap11}). Note that
if $u$ is lower semi-bounded or $e^u\in D({\cal E})_e$ (cf.
\cite[Example 3.4 (iii)]{CSM07}), then $(F^u_t)_{t\ge 0}$ is
locally integrable on $[0, \zeta)$ for q.e. $x\in E$.
\end{rem}

\begin{rem}
If $(\mathcal{E},D(\mathcal{E}))$ is a symmetric Dirichlet form,
then the assumption of Theorem \ref{th1.1} is automatically
satisfied. Note that $(P^{u}_{t})_{t\geq0}$ is symmetric on
$L^{2}(E;m)$. If $(P^{u}_{t})_{t\geq0}$ is strongly continuous,
then (\ref{dog77})
holds (cf. \cite[Remark 1.6(ii)]{C09}). Therefore, the following three assertions are equivalent:\\
(i) $(Q^u,D(\mathcal{E})_{b})$ is lower semi-bounded.\\
(ii) There exists a constant $\alpha_0\ge 0$ such that
$\|P^{u}_{t}\|_2\leq e^{\alpha_0 t}$ for $t>0$.\\
(iii) $(P^{u}_{t})_{t\geq0}$ is strongly continuous on
$L^{2}(E;m)$.
\end{rem}

\begin{rem}
Denote by $S$ the set of all smooth measures on
$(E,{\mathcal{B}}(E))$. Let $\mu=\mu_{1}-\mu_{2}\in S-S$,
$(A^{1}_{t})_{t\ge 0}$ and $(A^{2}_{t})_{t\ge 0}$ be PCAFs with
Revuz measures $\mu_{1}$ and $\mu_{2}$, respectively. Define
$$\bar{P}^A_{t}f(x)=E_{x}[e^{A^2_{t}-A^1_t}f(X_{t})],\ \ f\ge 0\ {\rm and}\ t\ge 0,$$
and \begin{eqnarray*}
\left\{
\begin{array}{l}{\cal
E}^{\mu}(f,g):=\mathcal{E}(f,g)+\int_Efgd\mu,\\
f,g\in D({\cal E}^{\mu}):=\{w\in D(\mathcal{E})|w\ {\rm is}\
(\mu_1+\mu_2)-{\rm square\ integrable}\}.
\end{array}
\right.
\end{eqnarray*}
Then, by a localization argument similar to that used in the proof
of Theorems \ref{th1.1} and \ref{th1.2} (cf. also \cite{C06}), we
can show that
the following two conditions are equivalent:\\
(i) There exists a constant $\alpha_{0}\geq0$ such that
$$
  {\cal
E}^{\mu}(f,f)\geq -\alpha_0(f,f)_m,\ \ \forall f\in D({\cal
E}^{\mu}).
$$
 (ii) There exists a constant $\alpha_{0}\geq0$ such that
$$
\|\bar{P}^{A}_{t}\|_2\leq e^{\alpha_{0}t},\ \ \forall t>0.$$
Furthermore, if one of these conditions holds, then the semigroup
$(\bar{P}^{A}_{t})_{t\geq0}$ is strongly continuous on
$L^{2}(E;m)$.

This result generalizes the corresponding results of \cite{AM1}
and \cite{C07}. Note that, similar to Theorems \ref{th1.1} and
\ref{th1.2}, it is not necessary to assume that the bilinear form
$({\cal E}^{\mu}, D({\cal E}^{\mu}))$ satisfies the sector
condition.
\end{rem}

\section[short title]{Examples}
\emph{\emph{\textbf{Example 3.1.}}}\ \ In this example, we study
the generalized Feynman-Kac semigroup for the non-symmetric
Dirichlet form given in \cite[II, 2 d)]{MR92}.

Let $d\geq3$, $U$ be an open set of $\mathbf{R^d}$ and $m=dx$, the
Lebesgue measure on $U$. Let $a_{ij}\in L^1_{\rm loc}(U;dx)$,
$1\le i,j\le d$, $b_i,d_i\in L^d_{\rm loc}(U;dx)$, $d_{i}-b_{i}\in
L^{d}(U;dx)\cup L^{\infty}(U;dx)$, $1\leq i\leq d$,  $c\in
L^{d/2}_{\rm loc}(U;dx)$. Define for $f,g\in C^{\infty}_{0}(U)$
\begin{eqnarray*}
\mathcal{E}(f,g) &=&\sum_{i,j=1}^{d}\int_U\frac{\partial
f}{\partial
     x_{i}}\frac{\partial g}{\partial
     x_{j}}a_{ij}dx+\sum_{i=1}^{d}\int_U
     f\frac{\partial g}{\partial x_{i}}d_{i}dx\\
\nonumber
   &&+\sum_{i=1}^{d}\int_U
     \frac{\partial f}{\partial x_{i}}gb_{i}dx+\int_U fgcdx.
\end{eqnarray*}

Denote $\tilde{a}_{ij}:=\frac{1}{2}(a_{ij}+a_{ji})$ and
$\check{a}_{ij}:=\frac{1}{2}(a_{ij}-a_{ji})$, $1\le i,j\le d$.
Suppose that the following conditions hold

\noindent \emph{(C1)} There exists $\gamma\in (0,\infty)$ such
that $ \sum_{i,j=1}^d
\tilde{a}_{ij}\xi_i\xi_j\ge\gamma\sum_{i=1}^d|\xi_i|^2,\ \ \forall
\xi=(\xi_1,\dots,\xi_d)\in\mathbf{R^d} $.

\noindent \emph{(C2)} $|\check{a}_{ij}|\le M\in (0,\infty)\ \ {\rm
for}\ 1\le i,j\le d$.

\noindent \emph{(C3)}
$cdx-\sum^{d}_{i=1}\frac{\partial{d_{i}}}{\partial x_{i}}\geq0$
and $cdx-\sum^{d}_{i=1}\frac{\partial{b_{i}}}{\partial
x_{i}}\geq0$ (in the sense of Schwartz distributions, i.e.,
$\int_U(cf+\sum^{d}_{i=1}d_{i}\frac{\partial{f}}{\partial
x_{i}})dx,
\int_U(cf+\sum^{d}_{i=1}b_{i}\frac{\partial{f}}{\partial x_{i}})dx
\geq0$ for all $f\in C^{\infty}_{0}(U)$ with $f\ge 0$).

\noindent Then $(\mathcal{E}, C^{\infty}_{0}(U))$ is closable and
its closure $(\mathcal{E},D(\mathcal{E}))$ is a regular Dirichlet
form on $L^{2}(U;dx)$ (see \cite[II, Proposition 2.11]{MR92}).

Let $u\in C^{\infty}_{0}(U)$. Then, for $f\in C^{\infty}_{0}(U)$,
we have
\begin{eqnarray*}
Q^u(f,f)
&=&\mathcal{E}(f,f)+\mathcal{E}(u,f^2)\\
&=&\sum_{i,j=1}^{d}\int_U\frac{\partial f}{\partial
    x_{i}}\frac{\partial f}{\partial
    x_{j}}a_{ij}dx+\int_Uf^2\left(c(1+u)+\sum_{i=1}^d\frac{\partial u}{\partial
    x_{i}}b_i\right)dx\\
&&+\int_U\sum_{i=1}^{d}\frac{\partial f^2}{\partial
    x_{i}}\left(\frac{d_i+b_i}{2}+ud_i+\sum_{j=1}^{d}\frac{\partial
u}{\partial
    x_{j}}a_{ji}\right)dx.
\end{eqnarray*}
Suppose that the following condition holds

\noindent \emph{(C4)} There exists a constant $\alpha_{0}\geq0$
such that
$$
\left(\alpha_0+c(1+u)+\sum_{i=1}^d\frac{\partial u}{\partial
    x_{i}}b_i\right)dx-\sum_{i=1}^{d}\frac{\partial (\frac{d_i+b_i}{2}+ud_i+\sum_{j=1}^{d}\frac{\partial
u}{\partial
    x_{j}}a_{ji})}{\partial
    x_{i}}\ge 0
$$
in the sense of Schwartz distribution.

\noindent Then $Q^{u}(f,f)\geq -\alpha_0(f,f)$ for any $f\in
C^{\infty}_{0}(U)$ and thus for any $f\in D({\cal E})_b$ by
approximation.

Let $X$ be a Hunt process associated with
$(\mathcal{E},D(\mathcal{E}))$ and $(P^{u}_{t})_{t\geq0}$ be the
generalized Feynman-Kac semigroup induced by $u$. Then, by Theorem
\ref{th1.1} or Theorem \ref{th1.2}, $(e^{-\alpha_0
t}P^{u}_{t})_{t\geq0}$ is a strongly continuous contraction
semigroup on $L^{2}(U;dx)$. \vskip 0.4cm \noindent
\emph{\emph{\textbf{Example 3.2.}}}\ \  In this example, we study
the generalized Feynman-Kac semigroup for the non-symmetric
Dirichlet form given in \cite[II, 3 e)]{MR92}.

Let $E$ be a locally convex topological real vector space which is
a (topological) Souslin space. Let $m:=\mu$ be a finite positive
measure on ${\cal B}(E)$ such that supp\,$\mu=E$. Let $E'$ denote
the dual of $E$ and $_{E'}\langle,\rangle_E:E'\times
E\rightarrow\mathbf{R}$ the corresponding dualization. Define
$$
\mathcal{F}C^{\infty}_{b}:=\{f(l_{1},\dots,l_{m})|m\in \mathbf{N},
f\in C^{\infty}_{b}(\mathbf{R^{m}}),l_{1},\dots,l_{m}\in
E^{'}\}.$$ Assume that there exists a separable real Hilbert space
$(H, \langle,\rangle_{H})$ densely and continuously embedded into
$E$. Identifying $H$ with its dual $H^{'}$ we have that
$$E'\subset H\subset E \mbox{\ densely and continuously},$$
and $_{E'}\langle,\rangle_{E}$ restricted to $E'\times H$
coincides with $\langle,\rangle_{H}$. For $f\in
\mathcal{F}C^{\infty}_{b}$ and $z\in E$, define $\nabla u(z)\in H
$ by
$$
\langle\nabla u(z),h\rangle_{H}=\frac{\partial u}{\partial h}(z),\
\  h\in H.
$$

Let $({\cal E}_{\mu},\mathcal{F}C^{\infty}_{b})$, defined by $$
{\cal E}_{\mu}(f,g)=\int_E\langle\nabla f,\nabla g\rangle_Hd\mu,\
\ f,g\in \mathcal{F}C^{\infty}_{b},
$$
be closable on $L^2(E;\mu)$ (cf. \cite[II, Proposition 3.8 and
Corollary 3.13]{MR92}). Let $\mathcal{L}_{\infty}(H)$ denote the
set of all bounded linear operators on $H$ with operator norm $\|\
\|$. Suppose $z\rightarrow A(z), z\in E$, is a map from $E$ to
$\mathcal{L}_{\infty}(H)$ such that $z\rightarrow \langle
A(z)h_{1},h_{2}\rangle_{H}$ is $\mathcal{B}(E)$-measurable for all
$h_{1},h_{2}\in H$. Furthermore, suppose that the following
conditions hold

\noindent \emph{(C1)} There exists $\gamma\in (0,\infty)$ such
that $\langle A(z)h,h\rangle_{H}\geq\gamma\|h\|^{2}_{H}$ for all
$h\in H$.

\noindent \emph{(C2)} $\|\tilde{A}\|_{\infty}\in L^{1}(E;\mu)$ and
$\|\check{A}\|_{\infty}\in L^{\infty}(E;\mu)$, where
$\tilde{A}:=\frac{1}{2}(A+\hat{A}),
\check{A}:=\frac{1}{2}(A-\hat{A})$ and $\hat{A}(z)$ denotes the
adjoint of $A(z), z\in E$.

\noindent \emph{(C3)} Let $c\in L^{\infty}(E,\mu)$ and $b,d\in
L^{\infty}(E\rightarrow H;\mu)$ such that for
$u\in\mathcal{F}C^{\infty}_{b}$ with $u\geq0$
$$\int_E(\langle d,\nabla u\rangle_{H}+cu)d\mu\geq0, \int_E(\langle b,\nabla
u\rangle_{H}+cu)d\mu\geq0.
$$

Define for $f,g\in\mathcal{F}C^{\infty}_{b}$
\begin{eqnarray*}
\mathcal{E}(f,g)
 &=&\int_E\langle A\nabla f,\nabla
       g\rangle_{H}d\mu+\int_E f\langle d,\nabla g\rangle_{H}d\mu\\
 \nonumber &&+\int_E \langle b,\nabla f\rangle_{H}gd\mu+\int_E fgcd\mu.
\end{eqnarray*}
Then $(\mathcal{E},\mathcal{F}C^{\infty}_{b})$ is closable and its
closure $(\mathcal{E},D(\mathcal{E}))$ is a quasi-regular
Dirichlet form on $L^{2}(E,\mu)$ (see by \cite[II, 3 e)]{MR92}.

Let $u\in \mathcal{F}C^{\infty}_{b}$. Then, for $f\in
\mathcal{F}C^{\infty}_{b}$, we have
\begin{eqnarray*}
Q^{u}(f,f)
\nonumber&=&\mathcal{E}(f,f)+\mathcal{E}(u,f^2)\\
\nonumber&=&\int_E\langle A\nabla f,\nabla
             f\rangle_{H}d\mu+\int_E(c(1+u)+\langle b,\nabla u\rangle_H)f^2dx\\
&&+\int_E\left\langle\frac{d+b}{2}+ud+A\nabla u,\nabla
f^2\right\rangle_Hd\mu.
\end{eqnarray*}
Suppose that the following condition holds

\noindent \emph{(C4)} There exists a  constant $\alpha_{0}\geq0$
such that
$$
\int_E\left\{(\alpha_0+c(1+u)+\langle b,\nabla
u\rangle_H)f+\left\langle\frac{d+b}{2}+ud+A\nabla u,\nabla
f\right\rangle_H\right\}d\mu\ge 0
$$
for all $f\in \mathcal{F}C^{\infty}_{b}$ with $f\ge 0$.

\noindent Then $Q^{u}(f,f)\geq -\alpha_0(f,f)$ for any $f\in
\mathcal{F}C^{\infty}_{b}$ and thus for any $f\in D({\cal E})_b$
by approximation.

Let $X$ be a $\mu$-tight special standard diffusion process
associated with $(\mathcal{E},D(\mathcal{E}))$ and
$(P^{u}_{t})_{t\geq0}$ be the generalized Feynman-Kac semigroup
induced by $u$. Then, by Theorem \ref{th1.1}, $(e^{-\alpha_0
t}P^{u}_{t})_{t\geq0}$ is a strongly continuous contraction
semigroup on $L^{2}(E;\mu)$.


\end{document}